\documentclass[12pt]{amsart}

\usepackage{enumerate, amsmath, amsthm, amsfonts, amssymb, xy,  mathrsfs, graphicx, paralist}
\usepackage[usenames, dvipsnames]{xcolor}
\usepackage[margin=1in]{geometry} 
\usepackage[bookmarks, colorlinks=true, linkcolor=blue, citecolor=blue, urlcolor=blue]{hyperref}
\usepackage[boxsize=1em]{ytableau}
\usepackage{multicol}

\input xy
\xyoption{all}

\numberwithin{equation}{section}
\newtheorem{theorem}[equation]{Theorem}

\newtheorem{proposition}[equation]{Proposition}
\newtheorem{lemma}[equation]{Lemma}
\newtheorem{corollary}[equation]{Corollary}

\theoremstyle{definition}
\newtheorem{rmk}[equation]{Remark}
\newenvironment{remark}[1][]{\begin{rmk}[#1] \pushQED{\qed}}{\popQED \end{rmk}}
\newtheorem{eg}[equation]{Example}
\newenvironment{example}[1][]{\begin{eg}[#1] \pushQED{\qed}}{\popQED \end{eg}}
\newtheorem{defn}[equation]{Definition}
\newenvironment{definition}[1][]{\begin{defn}[#1]\pushQED{\qed}}{\popQED \end{defn}}

\newcommand{\rA}{\mathrm{A}}

\newcommand{\rB}{\mathrm{B}}

\newcommand{\bC}{\mathbf{C}}

\newcommand{\rC}{\mathrm{C}}

\newcommand{\rE}{\mathrm{E}}

\newcommand{\fG}{\mathfrak{G}}

\newcommand{\fH}{\mathfrak{H}}
\newcommand{\rH}{\mathrm{H}}

\newcommand{\fI}{\mathfrak{I}}

\newcommand{\bK}{\mathbf{K}}

\newcommand{\bO}{\mathbf{O}}

\newcommand{\rR}{\mathrm{R}}

\newcommand{\bS}{\mathbf{S}}

\newcommand{\cU}{\mathcal{U}}

\newcommand{\rU}{\mathrm{U}}

\newcommand{\bV}{\mathbf{V}}

\newcommand{\bW}{\mathbf{W}}

\newcommand{\bZ}{\mathbf{Z}}

\newcommand{\fg}{\mathfrak{g}}

\newcommand{\fl}{\mathfrak{l}}

\newcommand{\fn}{\mathfrak{n}}

\newcommand{\fp}{\mathfrak{p}}

\renewcommand{\phi}{\varphi}
\renewcommand{\emptyset}{\varnothing}

\newcommand{\arxiv}[1]{\href{http://arxiv.org/abs/#1}{{\tt arXiv:#1}}}

\makeatletter
\def\Ddots{\mathinner{\mkern1mu\raise\p@
\vbox{\kern7\p@\hbox{.}}\mkern2mu
\raise4\p@\hbox{.}\mkern2mu\raise7\p@\hbox{.}\mkern1mu}}
\makeatother

\DeclareMathOperator{\rank}{rank}

\DeclareMathOperator{\Sym}{Sym}

\DeclareMathOperator{\Tor}{Tor}

\newcommand{\GL}{\mathbf{GL}}

\newcommand{\Sp}{\mathbf{Sp}}

\newcommand{\fsl}{\mathfrak{sl}}
\newcommand{\fgl}{\mathfrak{gl}}
\newcommand{\fso}{\mathfrak{so}}
\newcommand{\fsp}{\mathfrak{sp}}

\DeclareMathOperator{\Rep}{Rep}
\newcommand{\bomega}{\mbox{\boldmath$\omega$}}

\title{Homology of analogues of Heisenberg Lie algebras}

\author{Steven V Sam}
\address{Department of Mathematics, University of California, Berkeley, CA}
\email{\href{mailto:svs@math.berkeley.edu}{svs@math.berkeley.edu}}
\urladdr{\url{http://math.berkeley.edu/~svs/}}

\date{July 10, 2015}

\thanks{The author was supported by a Miller research fellowship.}


\begin{document}

\maketitle

\begin{abstract}
We calculate the homology of three families of $2$-step nilpotent Lie (super)algebras associated with the symplectic, orthogonal, and general linear groups. The symplectic case was considered by Getzler and the main motivation for this work was to complete the calculations started by him. In all three cases, these algebras can be realized as the nilpotent radical of a parabolic subalgebra of a simple Lie algebra, and our first approach relies on a theorem of Kostant, but is otherwise elementary and involves combinatorics of Weyl groups and partitions which may be of independent interest. Our second approach is an application of (un)stable representation theory of the classical groups in the sense of recent joint work of the author with Snowden, which is shorter and more conceptual. 
\end{abstract}

\section{Introduction}

We work over the complex numbers $\bC$. Let $V$ be a symplectic vector space of dimension $2n$ (with symplectic form $\omega_V \colon \bigwedge^2 V \to \bC$) and let $E$ be a vector space of dimension $k$. Define a 2-step nilpotent Lie algebra
\[
\fH = \fH_V(E) = (E \otimes V) \oplus \Sym^2(E)
\]
with Lie bracket on pure tensors given by
\[
[(e \otimes v, x), (e' \otimes v', x')] = (0, \omega_V(e,e') vv')
\]
where $e,e' \in E$, $v,v' \in V$ and $x,x' \in \Sym^2(E)$. The bracket is compatible with the action of $\GL(E) \times \Sp(V)$ on $\fH$. 

In \cite{getzler}, the problem of calculating the Lie algebra homology of $\fH$ is raised. We state the result in Theorem~\ref{thm:main}, but it requires some combinatorial preliminaries.

Given a partition $\lambda = (\lambda_1, \dots, \lambda_n)$, let $\lambda^\dagger$ denote the transpose partition. This is best explained in terms of Young diagrams. If $\lambda = (5,3,2)$, then $\lambda^\dagger = (3,3,2,1,1)$:
\[
\lambda = \ydiagram{5,3,2}, \qquad \lambda^\dagger = \ydiagram{3,3,2,1,1}.
\]
We set $\ell(\lambda)$ to be the number of nonzero parts of $\lambda$, and $|\lambda| = \sum_i \lambda_i$. So $\ell(5,3,2) = 3$ and $|(5,3,2)| = 10$. The irreducible polynomial representations of $\GL(E)$ are indexed by partitions $\lambda$ with $\ell(\lambda) \le \dim E$. We denote them by $\bS_\lambda(E)$. See \cite[\S 6.1]{fultonharris} for details. The construction of $\bS_\lambda(E)$ makes sense without any restriction on $\ell(\lambda)$, but $\bS_\lambda(E) = 0$ whenever $\ell(\lambda) > \dim E$. Similarly, the irreducible polynomial representations of $\Sp(V)$ are indexed by partitions $\mu$ with $2\ell(\mu) \le \dim V$, and we denote them by $\bS_{[\mu]}(V)$. See \cite[\S 17.3]{fultonharris} for details. 

Given a partition $\lambda$ we will define $i_{2n}(\lambda) \in \bZ_{\ge 0} \cup \{\infty\}$ and $\tau_{2n}(\lambda)$ which is either a partition with $\ell(\tau_{2n}(\lambda)) \le n$, or is undefined. These definitions, in a different form, originally appeared in \cite[\S 2.4]{koiketerada}. They will be used to describe the homology of $\fH(E)$. There are several descriptions of the functions $i_{2n}$ and $\tau_{2n}$ (which can be found in \cite[\S 3.4]{lwood}); we give one of them now and another one in Definition~\ref{defn:modrule-weyl}.

\begin{definition}[Modification rule -- border strip version] \label{defn:modrule-border}
If $\ell(\lambda) \le n$ we put $i_{2n}(\lambda)=0$ and $\tau_{2n}(\lambda)=\lambda$.  Suppose $\ell(\lambda)>n$. A {\bf border strip} is a connected skew Young diagram containing no $2 \times 2$ square. Let $R_{\lambda}$ be the connected border strip of length $2(\ell(\lambda)-n-1)$ which starts at the first box in the final row of $\lambda$, if it exists. The box of $R_\lambda$ in the bottom row of $R_\lambda$ is its {\bf first box}, and the box at the top row is its {\bf last box}. If $R_{\lambda}$ exists, is non-empty, and $\lambda \setminus R_{\lambda}$ is a partition, then we put $i_{2n}(\lambda)=c(R_{\lambda})+i_{2n}(\lambda \setminus R_{\lambda})$ and $\tau_{2n}(\lambda)=\tau_{2n}(\lambda \setminus R_{\lambda})$, where $c(R_{\lambda})$ denotes the number of columns that $R_{\lambda}$ occupies; otherwise we put $i_{2n}(\lambda)=\infty$ and leave 
$\tau_{2n}(\lambda)$ undefined.
\end{definition}

\begin{example}
Set $n=1$ and $\lambda = (4,3,3,2,2,1,1)$. Then $2(\ell(\lambda)-n-1) = 10$. We have shaded in the border strip $R_\lambda$ of length $10$ in the Young diagram of $\lambda$:
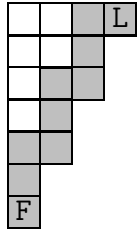
\[
\ytableaushort
{\none \none \none {\tt L}, \none, \none, \none, \none, \none, {\tt F}}
*[*(white)]{2,2,1,1}
*[*(lightgray)]{4,3,3,2,2,1,1}
\]
The first box is marked with an {\tt F} and the last box is marked with an {\tt L}. From this, we see that $i_2(\lambda) = 4 + i_2(2,2,1,1)$. Repeating the process, the next border strip to remove has length $4$ and $i_2(2,2,1,1) = 2 + i_2(2)$, and the result is the partition $(2)$, which has length $\le 1$, so we are done. The conclusion is that $i_2(\lambda) = 4 + 2 = 6$ and $\tau_2(\lambda) = (2)$.
\end{example}

\begin{theorem} \label{thm:main}
We have an isomorphism of $\GL(E) \times \Sp(V)$-modules
\[
\rH_i(\fH_V(E); \bC) = \bigoplus_{\substack{\lambda\\ |\lambda| - i_{2n}(\lambda) = i}} \bS_{\lambda^\dagger}(E) \otimes \bS_{[\tau_{2n}(\lambda)]}(V).
\]
\end{theorem}

In this paper we will give two proofs of this theorem. The first proof is in \S\ref{sec:proof1} and relies on a general theorem of Kostant and a simplification of the related combinatorics. This proof has the advantage of being elementary. A second proof is given in \S\ref{sec:proof2} and relies on recent joint work of the author with Andrew Snowden \cite{infrank}. This proof has the advantage of being shorter and more conceptual, though it is less elementary. It also reveals some extra structure of the problem as we consider the limit $\dim(V) \to \infty$. The second proof easily generalizes to calculate the homology of some $2$-step nilpotent Lie algebras and Lie superalgebras which can be considered as orthogonal and general linear versions of the Lie algebra $\fH(E)$. This will be done in \S\ref{sec:complements}.

\begin{remark}
\begin{compactenum}[1.]
\item From our discussion above, $\bS_{\lambda^\dagger}(E) = 0$ as soon as $\lambda_1 > \dim E = k$. Also, the largest possible size border strip that can be removed from a partition $\lambda$ is of size $\ell(\lambda) + \lambda_1 - 1$, so if $i_{2n}(\lambda) < \infty$, then we must have $\ell(\lambda) \le k + 2n + 1$ by Definition~\ref{defn:modrule-border}. So the partitions $\lambda$ in the above sum are limited to those that fit into a $(k + 2n + 1) \times k$ rectangle. In fact, we will see in \eqref{eqn:sizeWP} that there are exactly $2^k \binom{n+k}{k}$ partitions $\lambda$ appearing in the total homology.

\item The formulation of the calculation in \cite{getzler} is to keep $V$ fixed and to treat each $\rH_i(\fH(E); \bC)$ as a polynomial functor in $E$. Our formulation in Theorem~\ref{thm:main} does exactly this. \qedhere
\end{compactenum}
\end{remark}

\begin{example}
Take $n=k=2$. There are $24$ terms that appear in the homology of $\fH$ in this case. We list them below. The first entry is the homological degree, and the second entry is of the form $(-\mu_2,-\mu_1, \lambda_1, \lambda_2)$ to denote the representation $\bS_\mu(E) \otimes \bS_{[\lambda]}(V)$.
\begin{multicols}{4}
\noindent $0\quad(0, 0, 0, 0)\\
1\quad(0, -1, 1, 0)\\
2\quad(-1, -1, 2, 0)\\
2\quad(0, -2, 1, 1)\\
3\quad(-1, -2, 2, 1)\\
3\quad(0, -4, 1, 1)\\
4\quad(-2, -2, 2, 2)\\
4\quad(-1, -4, 2, 1)\\
4\quad(0, -5, 1, 0)\\
5\quad(-2, -4, 2, 2)\\
5\quad(-1, -5, 2, 0)\\
5\quad(0, -6, 0, 0)\\
6\quad(-3, -5, 2, 2)\\
6\quad(-2, -6, 2, 0)\\
6\quad(-1, -7, 0, 0)\\
7\quad(-5, -5, 2, 2)\\
7\quad(-3, -6, 2, 1)\\
7\quad(-2, -7, 1, 0)\\
8\quad(-5, -6, 2, 1)\\
8\quad(-3, -7, 1, 1)\\
9\quad(-6, -6, 2, 0)\\
9\quad(-5, -7, 1, 1)\\
10\quad(-6, -7, 1, 0)\\
11\quad(-7, -7, 0, 0)$
\end{multicols}
\noindent We remark on the Poincar\'e duality present in this calculation: in general, one has $\rH_i(\fg; W) \cong \rH_{\dim \fg - i}(\fg; W^* \otimes \det \fg)^*$ for any $\fg$-module $W$, and $\det \fH = (\det E)^7$ as a representation of $\GL(E) \times \Sp(V)$.
\end{example}

\begin{remark}
We have $\tau_{2n}(\lambda) = \emptyset$ if and only if $\bS_{\lambda}(E)$ appears in the minimal free resolution of the ideal $I_n$ of $2(n+1) \times 2(n+1)$ Pfaffians of the generic skew-symmetric matrix over $A = \Sym(\bigwedge^2 E)$ (assuming that $\dim E \ge \ell(\lambda)$). More precisely, there is a $\GL(E)$-action on $I_n$ and we ask that $\bS_\lambda(E)$ is a subrepresentation of $\Tor_\bullet^A(A/I_n, \bC)$. See \cite[Remark 3.7]{lwood} for details. This set is described in \cite[\S 6.4]{weyman}: every such partition is of the form
\[
(s + \alpha_1, \dots, s + \alpha_s, s, \dots, s, \alpha^\dagger_1, \dots, \alpha^\dagger_r)
\]
where, in the middle part, $s$ is repeated $2n+1$ times, and $\alpha$ is any partition with $\ell(\alpha) \le s$.
\end{remark}

\section{Proof of Theorem~\ref{thm:main} via Kostant's theorem} \label{sec:proof1}

We give a proof of Theorem~\ref{thm:main} using a theorem of Kostant and the fact that the Lie algebra $\fH$ is the nilpotent radical of a parabolic subalgebra of a semisimple Lie algebra. First we state Kostant's theorem in \S\ref{sec:kostant} and why it is relevant to our case in \S\ref{sec:parabolic}. Then we make all of the combinatorics explicit in \S\ref{sec:weylgroup} and \S\ref{sec:modrule} and finally give the proof at the end of the section.

\subsection{Kostant's theorem} \label{sec:kostant}

We state Kostant's theorem in this section (see \cite[Theorem 5.14]{kostant} or \cite[Theorem 3.2.7]{kumar} for a more general version). For simplicity, we only state it for finite-dimensional semisimple Lie algebras. For a review of the material in this section, we refer to \cite[Chapter 1]{kumar}. However, we will only need some specific cases of this general theory, and it will be made explicit in the following sections, so it is not logically necessary for the reader to be familiar with the general setting.

Let $\fg$ be a finite-dimensional semisimple Lie algebra and let $\fp \subset \fg$ be a parabolic subalgebra with nilpotent radical $\fn$ and Levi subalgebra $\fl$. Let $W$ be the Weyl group of $\fg$ and let $W_P$ be the Weyl group of $\fl$. Then $W_P \subset W$ is a parabolic subgroup. Let $\ell \colon W \to \bZ_{\ge 0}$ be the length function on $W$. Let $\rho$ be the sum of the fundamental weights of $\fg$. For a weight $\lambda$ of $\fg$ and an element $w \in W$, define 
\[
w \bullet \lambda = w(\lambda + \rho) - \rho.
\]
In each left coset of $W/W_P$, there is a unique minimal length representative, and we denote this set by $W^P$. If $\lambda$ is a dominant weight, then $w^{-1} \bullet \lambda$ restricts to a dominant weight of $\fl$ if and only if $w \in W^P$. For a dominant integral weight $\lambda$ of $\fg$, let $L(\lambda)$ be the irreducible $\fg$-module with highest weight $\lambda$, and similarly for a dominant integral weight $\mu$ of $\fl$, let $L_{\fl}(\mu)$ be the irreducible $\fl$-module with highest weight $\fl$. In both cases, $L(\lambda)$ and $L_\fl(\mu)$ are finite-dimensional. We think of $L(\lambda)$ as an $\fn$-module through the inclusion $\fn \subset \fg$.

\begin{theorem}[Kostant] \label{thm:kostant}
Let $\lambda$ be a dominant integral weight of $\fg$. We have an isomorphism of $\fl$-modules
\[
\rH_i(\fn; L(\lambda)^*) \cong \bigoplus_{\substack{w \in W^P\\ \ell(w) = i}} L_{\fl}(w^{-1} \bullet\lambda)^*.
\]
\end{theorem}

We remark that in \cite[Theorem 3.2.7]{kumar}, the result is stated for the nilpotent radical $\fn^-$ of the opposite parabolic subalgebra. This is isomorphic to $\fn$ as a Lie algebra, but is the dual of $\fn$ from the perspective of the Levi subalgebra $\fl$, and also the representation $L(\lambda)$ restricted $\fn^-$ becomes the representation $L(\lambda)^*$ restricted to $\fn$, which is why we have added the duals above.

\subsection{Some parabolic subalgebras} \label{sec:parabolic}

Put a symplectic form on $U = V \oplus E \oplus E^*$ by
\[
\omega((v,e,\phi), (v',e', \phi')) = \omega_V(v,v') + \phi'(e) - \phi(e').
\]
We have a $\bZ$-grading on $\fsp(U)$ which is supported on $[-2,2]$:
\[
\fsp(U) = \Sym^2(E^*) \oplus (E^* \otimes V^*) \oplus (\fgl(E) \times \fsp(V)) \oplus (E \otimes V) \oplus \Sym^2(E).
\]
If $\dim E = k$, then the Dynkin diagram of $\fsp(U)$ is of type $\rC_{n+k}$ and this grading is associated with the $k$th node in Bourbaki notation. In particular, $\fsp(U)_{\ge 0}$ is a parabolic subalgebra and $\fsp(U)_{> 0} = \fH(E)$ is its nilpotent radical. The Levi subalgebra is $\fsp(U)_0$, and the Dynkin diagram of its semisimple subalgebra $\fsl(k) \times \fsp(2n)$ is of type $\rA_{k-1} \times \rC_n$. 

The Weyl group $W(\rC_N)$ of type $\rC_N$ is the group of signed permutations on $N$ letters, so has size $2^N N!$, and the Weyl group $W(\rA_N)$ of type $\rA_N$ is the group of permutations on $N+1$ letters, so has size $(N+1)!$. In particular, 
\begin{align} \label{eqn:sizeWP}
|W^P| = \frac{|W(\rC_{n+k})|}{|W(\rA_{k-1} \times \rC_n)|} = \frac{2^{n+k}(n+k)!}{k! \cdot 2^nn!} = 2^k \binom{n+k}{k}.
\end{align}
So we can calculate the Lie algebra homology of $\fH(E)$ via Kostant's theorem (see \S\ref{sec:kostant}).

\subsection{Weyl groups of classical groups} \label{sec:weylgroup}

A weight of $\fsp(U)$ is a sequence $\lambda \in \bC^{n+k}$ and it is a dominant integral weight precisely when $\lambda_1 \ge \cdots \ge \lambda_{n+k} \ge 0$ and $\lambda_i \in \bZ_{\ge 0}$ for $i=1,\dots,n+k$. We have 
\begin{align} \label{eqn:rho}
\rho = (n+k, n+k-1, \dots, 2, 1) \in \bZ_{\ge 0}^{n+k}.
\end{align}
The Weyl group $W$ of $\fsp(U)$ is the group of signed permutations which acts on sequences of length $n+k$ in the obvious way. The simple reflections that generate $W$ are as follows: for $i=1,\dots,n+k-1$, $s_i$ is the transposition that switches positions $i$ and $i+1$, and $s_{n+k}$ negates the last entry. By definition, the length of an element $w \in W$ is the minimal number of $s_i$ needed to generate $w$. We will need something more explicit. To describe the length function on $W$, consider the sequence $w(\rho)$. Define the following statistics:
\begin{align*}
{\rm inv}(w) &= \#\{1 \le i < j \le n+k \mid w(\rho)_i < w(\rho)_j \},\\
{\rm neg}(w) &= \#\{1 \le i \le n+k \mid w(\rho)_i < 0\},\\
{\rm nsp}(w) &= \#\{1 \le i < j \le n+k \mid w(\rho)_i + w(\rho)_j < 0\},
\end{align*}
which we call inversions, negative entries, and negative pairs. Then 
\begin{align} \label{eqn:BClength}
\ell(w) = {\rm inv}(w) + {\rm neg}(w) + {\rm nsp}(w).
\end{align}
For a proof, see \cite[Proposition 8.1.1]{bjornerbrenti}. We remark that we have to define the statistics above in terms of $w(\rho) = (w(n+k), \dots, w(1))$ instead of $(w(1), \dots, w(n+k))$ because \cite{bjornerbrenti} use the reflection $s_0$ that negates the first entry instead of our $s_{n+k}$. 

In particular, for $w \in W$, we have $w^{-1} \in W^P$ if and only if 
\begin{align} \label{eqn:WP}
w(\rho)_1 > w(\rho)_2 > \cdots > w(\rho)_k \quad \text{ and } \quad w(\rho)_{k+1} > w(\rho)_{k+2} > \cdots > w(\rho)_{k+n} > 0.
\end{align}

\subsection{Relation to modification rule} \label{sec:modrule}

We first give an alternative version of Definition~\ref{defn:modrule-border}. We will not use it, but it may clear up some of the mystery behind why Definition~\ref{defn:modrule-border} is related to the formula in Theorem~\ref{thm:kostant}.

\begin{definition}[Modification rule -- Weyl group version] \label{defn:modrule-weyl}
Let $\cU$ be the set of integer sequences $(\dots, a_2, a_1)$. For $i>0$, let $s_i$ be the involution on $\cU$ that swaps $a_i$ and $a_{i+1}$. Let $s_0$ be the involution that negates $a_1$. We let $W(\rB\rC_\infty)$ be the group generated by the $s_i$, for $i \ge 0$.  Then $W(\rB\rC_\infty)$ is a Coxeter group of type $\rB\rC_{\infty}$, so it is equipped with a length function $\ell \colon W(\rB\rC_\infty) \to \bZ_{\ge 0}$.  Let $\sigma=(\dots, -(n+2), -(n+1))$. Define a new action of $W(\rB\rC_\infty)$ on $\cU$ by $w \bullet \lambda=w(\lambda+\sigma)-\sigma$. Given a partition $\lambda$, we interpret it as an element of $\cU$ via $(\dots,\lambda_2, \lambda_1)$. Then exactly one of the following two possibilities hold:
\begin{compactitem}
\item There exists a unique element $w \in W(\rB\rC_\infty)$ such that $w \bullet \lambda^{\dag}=\mu^{\dag}$ is a partition and $\ell(\mu) \le n$. We then put $i_{2n}(\lambda)=\ell(w)$ and $\tau_{2n}(\lambda)=\mu$.
\item There exists a non-identity element $w \in W(\rB\rC_\infty)$ such that $w \bullet \lambda^{\dag}=\lambda^{\dag}$.  We then put $i_{2n}(\lambda)=\infty$ and leave $\tau_{2n}(\lambda)$ undefined. \qedhere
\end{compactitem}
\end{definition}

For a proof that Definition~\ref{defn:modrule-border} and Definition~\ref{defn:modrule-weyl} agree, see \cite[Proposition 3.5]{lwood}. We remark that there is a gap in the proof for showing that both definitions of $i_{2n}(\lambda)$ agree: the first part of the proof constructs an element $w \in W(\rB\rC_\infty)$ as a product of shorter elements and assumes that the length is additive. This is not proven, but follows from \eqref{eqn:BClength} and the following lemma.

\begin{lemma} \label{lem:inc-col}
Use notation as in Definition~\ref{defn:modrule-border}. Choose partitions $\mu$, $\nu$, $\lambda$ with $\nu = \mu \setminus R_\mu$ and $\lambda = \nu \setminus R_\nu$. Then $c(R_\mu) > c(R_\nu)$.
\end{lemma}

\begin{proof}
Let $h_\mu = |R_\mu| - c(R_\mu)$ and $h_\nu = |R_\nu| - c(R_\nu)$. Then $h_\mu+1$ is the number of rows that $R_\mu$ occupies and similarly for $h_\nu+1$. Since $R_\nu$ sits above $R_\mu$, if the last box in $R_\nu$ is in a row strictly lower than the row of the last box of $R_\mu$, we have $c(R_\mu) > c(R_\nu)$. So suppose that the last box in $R_\nu$ is in a row at least as high as the row of the last box of $R_\mu$, so that $\ell(\mu) - h_\mu \ge \ell(\nu) - h_\nu$, or equivalently, $h_\nu - h_\mu \ge \ell(\nu) - \ell(\mu)$. Then we have
\[
c(R_\mu) - c(R_\nu) = h_\nu - h_\mu + 2(\ell(\mu) - \ell(\nu)) \ge \ell(\mu) - \ell(\nu) > 0. \qedhere
\]
\end{proof}

Let $\lambda$ be a partition that fits into the $n \times k$ rectangle $R_{n,k}$, i.e., $\ell(\lambda) \le n$ and $\lambda_1 \le k$. By looking at the lower hull of the Young diagram of $\lambda$, we can record $\lambda$ as a lattice path from the bottom left corner of $R_{n,k}$ to the top right corner of $R_{n,k}$ that only uses steps in the up direction and the in right direction (see Example~\ref{eg:rectangle}). This gives us a sequence of ups and rights, let $\alpha_1 < \alpha_2 < \cdots < \alpha_n$ be the positions of the ups and let $\beta_1 < \cdots < \beta_k$ be the sequence of the rights. Recall the definition of $\rho$ from \eqref{eqn:rho}. Define $w_\lambda \in W$ by 
\begin{align} \label{eqn:wlambda}
w_\lambda(\rho) = (\beta_k, \dots, \beta_2, \beta_1, \alpha_n, \dots, \alpha_2, \alpha_1).
\end{align}
Then $w_\lambda^{-1} \in W^P$ by \eqref{eqn:WP} and $\ell(w_\lambda) = |\lambda|$ by \eqref{eqn:BClength}. Furthermore, we have
\[
w_\lambda(\rho) - \rho = (-\lambda_k^\dagger, \dots, -\lambda_2^\dagger, -\lambda_1^\dagger, \lambda_1, \lambda_2, \dots, \lambda_n).
\]
To see this, note that $\alpha_i - i = \#\{j \mid \alpha_i > \beta_j\}$, and this is just $\lambda_{n+1-i}$. Similarly, by reversing the situation (i.e., working in a transpose rectangle or just walking backwards in the lattice path) we see that $\beta_i - (n+i) = -\#\{j \mid \alpha_j > \beta_i\}$, and this is just $-\lambda_i^\dagger$.

\begin{example} \label{eg:rectangle}
Let $n=5$ and $k=6$ and consider $\lambda = (5,3,3,1)$. Then $\alpha = (1 < 3 < 6 < 7 < 10)$ and $\beta = (2 < 4 < 5 < 8 < 9 < 11)$:
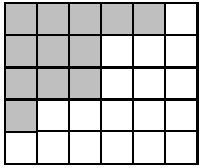
\[
\ydiagram[*(lightgray)]{5,3,3,1}*[*(white)]{6,6,6,6,6}
\]
so $w_\lambda(\rho) = (11,9,8,5,4,2, 10,7,6,3,1)$ and $w_\lambda(\rho)-\rho = (0,-1,-1,-3,-3,-4,5,3,3,1,0)$.
\end{example}

\begin{lemma} \label{lem:wmu}
Let $\mu$ be a partition with $\mu_1 \le k$ such that $\tau_{2n}(\mu) = \lambda$. There exists $w_\mu \in W$ such that $w_\mu^{-1} \in W^P$ and $\ell(w_\mu) = |\mu| - i_{2n}(\mu)$, and 
\begin{align} \label{eqn:mu}
w_\mu(\rho) - \rho = (-\mu_k^\dagger, \dots, -\mu_2^\dagger, -\mu_1^\dagger, \lambda_1, \lambda_2, \dots ,\lambda_n).
\end{align}
\end{lemma}

\begin{proof}
Let $t_0$ be the involution that negates the $k$th position. Also, for $i=1,\dots,k-1$, set $t_i = s_{k-i}$.

We will use the border strip version of the modification rule (Definition~\ref{defn:modrule-border}) in this proof. We will do induction on the number of border strips that we need to remove from $\mu$ in order to get $\lambda$. The base case is $\lambda = \mu$. In this case, we have defined $w_\lambda$ in \eqref{eqn:wlambda} and established its properties.

Now assume that $\ell(\mu) > n$ and set $\nu = \mu \setminus R_\mu$ where $R_\mu$ is a border strip of length $2(\ell(\mu) - n - 1)$ with $c = c(R_\mu)$ columns. By induction, there exists $w_\nu \in W$ with the desired properties. Set 
\[
w_\mu = t_0 t_1 t_2 \cdots t_{c-1} w_\nu.
\]
We have 
\[
w_\nu(\rho) = (\gamma_k, \dots, \gamma_1, \alpha_n, \dots, \alpha_1)
\]
where $\gamma_i -(n+ i) = -\nu^\dagger_i$ by \eqref{eqn:mu}. Then 
\[
w_\mu(\rho) = (\gamma_k, \dots, \gamma_{c+1}, \gamma_{c-1}, \dots, \gamma_1, -\gamma_c, \alpha_n, \dots, \alpha_1),
\]
and \eqref{eqn:mu} holds: by definition of border strip, we have 
\[
\nu_i^\dagger = \begin{cases} 
\mu_{i+1}^\dagger - 1 & \text{if $i<c$},\\
2n+2 + c-1 -\mu_1^\dagger & \text{if $i=c$},\\
\mu_i^\dagger & \text{if $i>c$}.
\end{cases}
\]
(The cases $i \ne c$ are clear, and the case $i=c$ can be deduced from this information and the fact that $|\mu| - |\nu| = 2\mu_1^\dagger - 2n - 2$.) By Lemma~\ref{lem:inc-col}, we have $\gamma_c = \beta_c$. So among the first $k$ entries, we have increased (relative to $w_\nu$) the number of negative entries by $1$, and increased the number of negative pairs by $c-1$. If we ignore $-\gamma_c$, then there are no new inversions, negative entries, or negative pairs. So we have to calculate the number of new inversions and negative pairs among $(-\gamma_c, \alpha_n, \dots, \alpha_1)$. By definition, each of these statistics is exactly $N_\mu = \#\{1 \le i \le n \mid \gamma_c > \alpha_i\}$. 

If $h_\mu$ is the index of the row where the last box of $R_\mu$ is, then $N_\mu = n - h_\mu + 1$. Also, the number of rows that $R_\mu$ occupies is $\ell(\mu) - h_\mu + 1$, so we see that $|R_\mu| = \ell(\mu) - h_\mu + c$ (the number of boxes in a border strip is the number of rows it occupies plus the number of columns it occupies minus $1$). In particular, 
\begin{align*}
\ell(w_\mu) &= \ell(w_\nu) + c + 2N_\mu\\
&= \ell(w_\nu) + c + 2n + 2 - 2h_\mu\\
&= \ell(w_\nu) + c + 2n + 2 + 2|R_\mu| - 2\ell(\mu) - 2c\\
&= \ell(w_\nu) - c + |R_\mu|.
\end{align*}
By induction, we have $\ell(w_\nu) = |\nu| - i_{2n}(\nu)$, and the two identities (that follow by definition) $|R_\mu| = |\mu| - |\nu|$ and $i_{2n}(\nu) + c = i_{2n}(\mu)$ give us that $\ell(w_\mu) = |\mu| - i_{2n}(\mu)$ as desired.
\end{proof}

\begin{lemma} \label{lem:exhaust}
The elements $w_\mu^{-1}$ constructed in Lemma~\ref{lem:wmu} as we range over all $\lambda \subseteq R_{n,k}$ exhaust all elements of $W^P$.
\end{lemma}

\begin{proof}
We have noted in \eqref{eqn:sizeWP} that $|W^P| = 2^k \binom{n+k}{k}$. There are $\binom{n+k}{k}$ choices for $\lambda \subseteq R_{n,k}$. Given $j \in \{1,\dots,k\}$, let $u_j = t_0t_1\cdots t_{j-1}$ using the notation of the proof of Lemma~\ref{lem:wmu}. Given a subset $\{i_1 < i_2 < \cdots i_r \} \subset \{1, \dots, k\}$, the element $u_{i_r} \cdots u_{i_2} u_{i_1} w_\lambda$ is of the form $w_\mu$ for some $\mu$ with $\mu_1 \le k$ and $\tau_{2n}(\mu) = \lambda$, by the reasoning in the proof of Lemma~\ref{lem:wmu}. So we can construct $2^k$ such $w_\mu$ for each choice of $\lambda$. This finishes the proof.

Alternatively, one could proceed by noting that the action of $W$ on the first $k$ entries of a sequence coincides with the action of $W(\rB\rC_\infty)$ on the last $k$ entries of elements of $\cU$ in Definition~\ref{defn:modrule-weyl} and that truncating $\rho$ gives $\sigma$.
\end{proof}

\subsection{The proof}

By \S\ref{sec:parabolic}, we can identify $\fH_V(E)$ with the nilpotent radical of a parabolic subalgebra $\fp \subset \fsp(2n+2k)$. Take $\lambda = 0$ in Theorem~\ref{thm:kostant}. Then we get the formula
\[
\rH_i(\fH_V(E); \bC) \cong \bigoplus_{\substack{w^{-1} \in W^P\\ \ell(w) = i}} L_{\fl}(w(\rho)-\rho)^*.
\]
where the condition $w^{-1} \in W^P$ is defined in \eqref{eqn:WP} and $\rho$ is defined in \eqref{eqn:rho}, and $\fl = \fgl(E) \times \fsp(V)$.

By Lemma~\ref{lem:exhaust}, the sequences of the form $w(\rho) - \rho$ for $w^{-1} \in W^P$ are all sequences of the form
\[
(-\mu_k^\dagger, \dots, -\mu_1^\dagger, \lambda_1, \dots, \lambda_n)
\]
where $\lambda$ is a partition and $\mu$ is a partition such that $\tau_{2n}(\mu) = \lambda$ and furthermore, $\ell(w) = |\mu| - i_{2n}(\mu)$. The corresponding representation $L_{\fl}(w(\rho)-\rho)^*$ of $\fl$ is $\bS_{\mu^\dagger}(E) \otimes \bS_{[\lambda]}(V)$. This finishes the proof.

\section{Proof of Theorem~\ref{thm:main} via stable representation theory} \label{sec:proof2}

Define $\bV = \bigcup_{n \ge 1} V_{2n}$ where $\dim(V_{2n}) = 2n$, each $V_n$ is equipped with a symplectic form, and we have compatible embeddings $V_{2n} \subset V_{2n+2}$. Then $\bV$ carries an action of the group $\Sp(\infty) = \bigcup_{n \ge 1} \Sp(2n)$, and we can define a category $\Rep(\Sp)$ which consists of subquotients of finite direct sum of tensor spaces $\bV^{\otimes m}$. The simple objects of this category are indexed by partitions $\lambda$ of all size, and we denote them by $\bS_{[\lambda]}(\bV)$. See \cite[\S 4.1]{infrank} for details. 

Similarly, we can define an orthogonal analogue of this category $\Rep(\bO)$. To distinguish notation, we denote the basic representation by $\bW$ and its simple objects are denoted $\bS_{[\lambda]}(\bW)$. Then $\Rep(\Sp)$ and $\Rep(\bO)$ are symmetric monoidal $\bC$-linear categories and there are nonzero maps $\bomega_\bV \colon \bigwedge^2 \bV \to \bC$ and $\bomega_\bW \colon \Sym^2(\bW) \to \bC$. Furthermore, there is an {\it anti-}symmetric monoidal equivalence of categories 
\[
\dagger \colon \Rep(\Sp) \to \Rep(\bO)
\]
\cite[Theorem 4.3.4]{infrank} which has the property $\bS_{[\lambda]}(\bV)^\dagger \cong \bS_{[\lambda^\dagger]}(\bW)$.

Let $E$ be a finite-dimensional vector space. Then we can define a Lie algebra object 
\[
\fH = \fH_\bV(E) = (E \otimes \bV) \oplus \Sym^2(E)
\]
in $\Rep(\Sp)$ using $\bomega_\bV$. Let $\rU(\fH)$ be its universal enveloping algebra. (By our finiteness conditions, $\rU(\fH)$ does not belong to $\Rep(\Sp)$, but it is graded-finite and belongs to a suitable enlargement of $\Rep(\Sp)$ where all of the results we use still apply.) Then $\rU(\fH)^\dagger$ is the universal enveloping algebra of a Lie superalgebra $\fH^\dagger = (E \otimes \bW) \oplus \Sym^2(E)$, which is an object of $\Rep(\bO)$.

Given a Lie algebra object $\fg$ in a symmetric monoidal $\bC$-linear category, we will let $\bK(\fg)_\bullet$ denote its Chevalley--Eilenberg complex \cite[\S 7.7]{weibel}.

\begin{proposition} \label{prop:stable-superhom}
$\Tor_i^{\rU(\fH)^\dagger}(\bC, \bC) = \bigoplus_{|\lambda|=i} \bS_\lambda(E) \otimes \bS_{[\lambda]}(\bW)$.
\end{proposition}

\begin{proof}
Write $\bW = \bigcup_{n \ge 1} W_n$ where $\dim(W_n) = n$, each $W_n$ is an orthogonal space, and we have compatible embeddings $W_n \subset W_{n+1}$. The Tor in question can be calculated as the homology of the complex $\bK(\fH^\dagger)_\bullet \otimes_{\rU(\fH)^\dagger} \bC$. The definition of $\bK(\fH^\dagger)$ is compatible with the sequence $(W_n)_{n \ge 1}$, so it is enough to fix $E$, and calculate the homology of $\bK(\fH_{W_n}(E)^\dagger)_\bullet \otimes \bC$ for $n \gg 0$.

We take $n \ge 2\dim(E)$. In this case, we claim that $\rU(\fH_{W_n}(E)^\dagger)$ is a Koszul algebra (using the natural grading $\deg(E \otimes W_n) = 1$ and $\deg(\Sym^2(E)) = 2$). Consider the polynomial algebra $A = \Sym(E \otimes W_n)$. Then we have quadratic polynomials 
\[
\Sym^2(E) \subset \Sym^2(E) \otimes \Sym^2(W_n) \subset \Sym^2(E \otimes W_n),
\]
and the ideal generated by them is a complete intersection (see for example \cite[Theorems 3.5, 3.8]{exceptional}). Let $B = A / \Sym^2(E)$ be the quotient. Then $B$ is a Koszul algebra, and its Koszul dual is $\rU(\fH_{W_n^*}(E^*)^\dagger)$ (see \cite[Example 10.2.3]{avramov}), and in particular, $\Tor_i^B(\bC, \bC) = \rU(\fH_{W_n}(E)^\dagger)_i$. Finally, the degree $i$ piece of $B$ has the decomposition $B_i = \bigoplus_{|\lambda|=i} \bS_\lambda(E) \otimes \bS_{[\lambda]}(W_n)$ \cite[Theorems 3.5, 3.8]{exceptional}. Since $\bS_{[\lambda]}(\bW) = \bigcup_n \bS_{[\lambda]}(W_n)$ \cite[(4.1.3)]{infrank}, we are done.
\end{proof}

\begin{corollary}
$\rH_i(\fH; \bC) = \Tor_i^{\rU(\fH)}(\bC, \bC) = \bigoplus_{|\lambda|=i} \bS_{\lambda^\dagger}(E) \otimes \bS_{[\lambda]}(\bV)$.  
\end{corollary}

\begin{proof}
The first equality is standard. The second equality comes from applying $\dagger$ to Proposition~\ref{prop:stable-superhom} and then reindexing $\lambda \mapsto \lambda^\dagger$.
\end{proof}

For each $n \ge 1$, we have specialization functors 
\[
\Gamma_{2n} \colon \Rep(\Sp) \to \Rep(\Sp(2n))
\]
which are left-exact, preserves the symmetric monoidal structure, and sends $\bV$ to $V_{2n}$ \cite[(4.4.4)]{infrank}.

\begin{proposition} \label{prop:hypercoh}
The hypercohomology $\rR^\bullet\Gamma_{2n}(\bK(\fH_\bV(E))_\bullet)$ calculates the Lie algebra homology of $\fH_{V_{2n}}(E)$, i.e., $\rR^i\Gamma_{2n}(\bK(\fH_\bV(E))_\bullet) = \rH_{-i}(\fH_{V_{2n}}(E); \bC)$.
\end{proposition}

\begin{proof}
By definition, the hypercohomology $\rR^\bullet \Gamma_{2n}(\bK(\fH_\bV(E))_\bullet)$ is calculated by finding an injective resolution of $\bK(\fH_\bV(E))_\bullet$, applying $\Gamma_{2n}$ to it, and then taking the homology of the resulting double complex. We have 
\[
\bK(\fH_\bV(E))_i = \bigwedge^i((E \otimes \bV) \oplus \Sym^2(E)) \otimes \rU(\fH_\bV(E))
\] 
and as an object of $\Rep(\Sp)$, the exterior power decomposes as a direct sum of Schur functors on $\bV$. Since $\rU(\fH_\bV(E)) = \Sym(\fH_\bV(E))$ as an object of $\Rep(\Sp)$, the same is true for $\rU(\fH_\bV(E))$. In particular, $\bK(\fH_\bV(E))_i$ is an injective object in $\Rep(\Sp)$ (see \cite[(4.2.9)]{infrank} for the statement for $\Rep(\bO)$, but both cases are proved in the same way). Since $\Gamma_{2n}$ is a symmetric monoidal functor, it takes Schur functors on $\bV$ to Schur functors on $V_{2n}$, so  $\Gamma_{2n}(\bK(\fH_{\bV}(E))_\bullet) = \bK(\fH_{V_{2n}}(E))_\bullet$, and we are done. (The minus sign comes from our homological indexing convention.)
\end{proof}

\begin{proof}[Proof of Theorem~\ref{thm:main}]
By \cite[\S 5.7.9]{weibel}, we have a hypercohomology spectral sequence
\[
\rE_2^{p,-q} = (\rR^p \Gamma_{2n})(\rH_q(\bK(\fH_\bV(E))_\bullet)) \Rightarrow \rR^{p-q} \Gamma_{2n}(\bK(\fH_\bV(E))_\bullet).
\]
By Proposition~\ref{prop:hypercoh}, the right hand side is $\rH_{q-p}(\fH_{V_{2n}}(E); \bC)$. Using \cite[Proposition 1.2, Theorem 3.6]{lwood}, we get
\[
\rR^p \Gamma_{2n}(\bS_{\lambda^\dagger}(E) \otimes \bS_{[\lambda]}(\bV)) = \begin{cases} \bS_{\lambda^\dagger}(E) \otimes \bS_{[\tau_{2n}(\lambda)]}(V_{2n}) & \text{if $p = i_{2n}(\lambda)$} \\ 0 & \text{else} \end{cases},
\]
which lets us calculate $\rE_2^{p,q}$. The differentials in the spectral sequence respect the action of $\GL(E)$, and no two terms have the same representation $\bS_{\lambda^\dagger}(E)$ appearing. So by Schur's lemma, all of the differentials in the $\rE_2$ page are $0$. Hence $\rE_2 = \rE_\infty$, and we get the desired result.
\end{proof}

\section{Complements} \label{sec:complements}

The two methods of calculation above can be used to calculate the Lie algebra homology of two more families of nilpotent Lie algebras, one associated with orthogonal Lie algebras and the other associated with general linear Lie algebras. The proof in \S\ref{sec:proof2} works with very few changes, so we will not repeat the arguments. We will state the relevant differences (also for the approach in \S\ref{sec:proof1}).

\subsection{Orthogonal version}

Let $V$ be an orthogonal space with orthogonal form $\omega_V$ and let $E$ be a vector space. We define a nilpotent Lie algebra
\[
\fI = \fI_V(E) = (E \otimes V) \oplus \bigwedge^2(E)
\]
with Lie bracket on pure tensors given by
\[
[(e \otimes v, x), (e' \otimes v', x')] = (0, \omega_V(e,e') v \wedge v')
\]
where $e, e' \in E$, $v, v' \in V$, and $x,x' \in \bigwedge^2(E)$. The Lie bracket is equivariant for the natural action of $\GL(E) \times \bO(V)$ on $\fI_V(E)$.

We can parametrize irreducible representations of $\bO(V)$ by partitions as in the symplectic case (see \cite[\S 4.1]{lwood} or \cite[\S 19.5]{fultonharris}). There is also a modification rule (see \cite[\S 4.4]{lwood}) and the analogue of Theorem~\ref{thm:main} holds without change. Set $m = \dim(V)$.

\begin{theorem} 
We have an isomorphism of $\GL(E) \times \bO(V)$-modules
\[
\rH_i(\fI_V(E); \bC) = \bigoplus_{\substack{\lambda\\ |\lambda| - i_{m}(\lambda) = i}} \bS_{\lambda^\dagger}(E) \otimes \bS_{[\tau_{m}(\lambda)]}(V).
\]
\end{theorem}

The calculation presented in \S\ref{sec:proof2} goes through with little change. The roles of $\Rep(\bO)$ and $\Rep(\Sp)$ are reversed, and the algebra $B$ in the proof of Proposition~\ref{prop:stable-superhom} is replaced by $\Sym(E \otimes W_{2n}) / \bigwedge^2(E)$ (where now $W_{2n}$ is a symplectic vector space). The fact that it is a complete intersection and its $\GL(E) \times \Sp(W)$-equivariant decomposition can be found in \cite[Theorem 3.1]{exceptional}. The specialization functors 
\[
\Gamma_m \colon \Rep(\bO) \to \Rep(\bO(m))
\]
\cite[(4.4.4)]{infrank} behave as in the symplectic case.

Alternatively, we can realize $\fI_V(E)$ as the nilpotent radical of the parabolic subalgebra of $\fso(U)$, obtained by marking the $k$th node of the Dynkin diagram ($k = \dim E$), and where $U = V \oplus E \oplus E^*$ with the orthogonal form 
\[
\omega_U((v,e,\phi), (v',e',\phi')) = \omega_V(v,v') + \phi'(e) + \phi(e'). 
\]
We state the necessary facts so that the reader can carry out the calculation in \S\ref{sec:kostant} if desired.

Choose $n$ so that $\dim(V) = 2n+1$ or $\dim(V) = 2n$. We use the subscripts ``odd'' and ``even'' to distinguish these two cases. A weight of $\fso(U)$ is a sequence $\lambda \in \bC^{n+k}$ and it is a dominant integral weight precisely when $\lambda_1 \ge \cdots \ge \lambda_{n+k-1} \ge |\lambda_{n+k}|$ (if $\dim(V) = 2n+1$, we have the additional restriction that $\lambda_{n+k} \ge 0$) and either $\lambda_i \in \bZ_{\ge 0}$ for all $i$, or $\lambda_i \in \frac{1}{2} + \bZ_{\ge 0}$ for all $i$. We have
\begin{align*}
\rho_{\rm odd} &= \frac{1}{2} (2n+2k-1, 2n+2k-3, \dots, 1),\\
\rho_{\rm even} &= (n+k-1, n+k-2, \dots, 1, 0).
\end{align*}

When $\dim(V) = 2n+1$, the Weyl group is the same as in the symplectic case. When $\dim(V) = 2n$, the Weyl group $W_{\rm even}$ of $\fso(U)$ is the group of signed permutations which acts on sequences of length $n+k$ with the restriction that the number of negative signs that appear is even. The simple reflections that generate $W_{\rm even}$ are as follows: for $i=1, \dots, n+k-1$, $s_i$ is the transposition that switches positions $i$ and $i+1$, and $s_{n+k}$ negates the last two positions and switches them. Given $w \in W_{\rm even}$, we can think of $w$ as a usual signed permutation so the statistics in \S\ref{sec:weylgroup} are defined. Then the length function on $W_{\rm even}$ is 
\[
\ell_{\rm even}(w) = {\rm inv}(w) + {\rm nsp}(w)
\]
\cite[Proposition 8.2.1]{bjornerbrenti}. Finally, for $w \in W_{\rm even}$, we have $w^{-1} \in W_{\rm even}^P$ if and only if 
\begin{align*}
&w(\rho_{\rm even})_1 > w(\rho_{\rm even})_2 > \cdots > w(\rho_{\rm even})_k \quad \text{ and }\\
&w(\rho_{\rm even})_{k+1} > \cdots > w(\rho_{\rm even})_{k+n-1} > |w(\rho_{\rm even})_{k+n}|.
\end{align*}

\subsection{General linear version} \label{sec:gl-version}

Let $V, E, F$ be vector spaces. We define a nilpotent Lie algebra
\[
\fG = \fG_V(E,F) = ((E \otimes V) \oplus (V^* \otimes F)) \oplus (E \otimes F)
\]
with Lie bracket on pure tensors given by
\[
[((e \otimes v, \phi \otimes f), x \otimes y), ((e' \otimes v', \phi' \otimes f'), x' \otimes y')] = (0, \phi'(v) e \otimes f' - \phi(v') e' \otimes f)
\]
where $e, e', x, x' \in E$, $v, v' \in V$, $\phi, \phi' \in V^*$, and $f,f',y,y' \in F$. The Lie bracket is equivariant for the natural action of $\GL(E) \times \GL(V) \times \GL(F)$ on $\fG_V(E,F)$.

We can parametrize irreducible (rational) representations of $\GL(V)$ by pairs of partitions $(\lambda,\mu)$ (see \cite[\S 5.1]{lwood}). There is also a modification rule (see \cite[\S 5.4]{lwood}) and the analogue of Theorem~\ref{thm:main} holds without change. Set $\dim(V) = n$.

\begin{theorem} 
We have an isomorphism of $\GL(E) \times \GL(V) \times \GL(F)$-modules
\[
\rH_i(\fG_V(E, F); \bC) = \bigoplus_{\substack{\lambda, \mu\\ |\lambda| + |\mu| - i_{n}(\lambda, \mu) = i}} \bS_{\lambda^\dagger}(E) \otimes \bS_{[\tau_{n}(\lambda, \mu)]}(V) \otimes \bS_{\mu^\dagger}(F).
\]
\end{theorem}

The calculation presented in \S\ref{sec:proof2} goes through with some changes. The equivalence $\dagger \colon \Rep(\Sp) \to \Rep(\bO)$ is replaced with an antisymmetric monoidal autoequivalence on $\Rep(\GL)$ that sends the simple object $\bS_{[\lambda, \mu]}(\bV)$ to $\bS_{[\lambda^\dagger, \mu^\dagger]}(\bV)$ \cite[Theorem 3.3.8]{infrank}.

The algebra $B$ in the proof of Proposition~\ref{prop:stable-superhom} is replaced by $\Sym((E \otimes V_{n}) \oplus (V_n^* \otimes F)) / (E \otimes F)$. This is a complete intersection whenever $n \ge \dim(E) + \dim(F)$ \cite[Lemma 5.3]{lwood} and its $\GL(E) \times \GL(V) \times \GL(F)$-equivariant decomposition can be found in \cite[\S 5.2]{lwood}. The specialization functors 
\[
\Gamma_n \colon \Rep(\GL) \to \Rep(\GL(n))
\]
\cite[(3.4.3)]{infrank} behave as in the symplectic case.

Alternatively, we can realize $\fG_V(E,F)$ as the nilpotent radical of the parabolic subalgebra of $\fgl(U)$, obtained by marking the $k$th and $(k+n)$th nodes of the Dynkin diagram ($k = \dim E$), and where $U = E \oplus V \oplus F$. We state the necessary facts so that the reader can carry out the calculation in \S\ref{sec:kostant} if desired.

A weight of $\fgl(U)$ is a sequence $\lambda \in \bC^{n+k+\ell}$ ($\ell = \dim F$) and it is a dominant integral weight precisely when $\lambda_1 \ge \cdots \ge \lambda_{n+k+\ell}$ and $\lambda_i \in \bZ$ for all $i$. We have
\begin{align*}
\rho = (n+k+\ell-1, n+k+\ell-2, \dots, 1, 0).
\end{align*}

The Weyl group is the symmetric group of all permutations acting on sequences of length $n+k+\ell$. The simple reflections that generate $W$ are as follows: for $i=1, \dots, n+k+\ell-1$, $s_i$ is the transposition that switches positions $i$ and $i+1$. Then the length function on $W$ is 
\[
\ell(w) = {\rm inv}(w).
\]
Finally, for $w \in W$, we have $w^{-1} \in W^P$ if and only if 
\[
w(\rho)_1 > \cdots > w(\rho)_k \quad \text{ and } \quad w(\rho)_{k+1} > \cdots > w(\rho)_{k+n} \quad \text{ and } \quad w(\rho)_{k+n+1} > \cdots > w(\rho)_{k+n+\ell}.
\]

\subsection{Lie superalgebra versions}

In all of the cases, we were mostly ambivalent about the dimension of the auxiliary vector space $E$ (and $F$ in \S\ref{sec:gl-version}). In fact, we can think about the space $V$ as being fixed and letting $\dim E$ and $\dim F$ grow to infinity. Alternatively, we can replace all of the representations $\bS_\lambda(E)$ by Schur functors $\bS_\lambda$ and all of the results above would carry over with no effort involved. Using the transpose duality \cite[\S 7.4]{expos} on the category of polynomial functors, we can replace all of the nilpotent Lie algebras studied in this paper with their corresponding Lie superalgebras (we have already done a little bit of this). All of the homology calculations are the same, except that we remove $\dagger$ from the notation.

\subsection{Recovering known results}

\subsubsection{Heisenberg Lie algebras}

When $\dim(E)=1$, the algebra $\fH = \fH_V(E)$ is what is usually called the Heisenberg Lie algebra. Its homology was calculated in \cite{heisenberg} where it was shown that $\dim_\bC \rH_i(\fH; \bC) = \binom{\dim V}{i} - \binom{\dim V}{i-2}$. We can get this from Theorem~\ref{thm:main} as follows. First, the term $\bS_{\lambda^\dagger}(E)$ is only nonzero (since $\dim(E)=1$) when $\lambda^\dagger = (i)$ (and so $\lambda = (1^i)$). Then $\tau_{2n}(1^i) = (1^i)$ and $i_{2n}(1^i) = 0$ if $0 \le i \le n$, and $\tau_{2n}(1^i) = (1^{2n+2-i})$ and $i_{2n}(1^i) = 1$ if $n+2 \le i \le 2n+2$. For all other $i$, we have $i_{2n}(1^i) = \infty$, and so
\[
\rH_i(\fH; \bC) = \begin{cases} \bS_{[1^i]}(V) & \text{if $0 \le i \le n$}\\
\bS_{[1^{2n+1-i}]}(V) & \text{if $n+1 \le i \le 2n+1$} \end{cases}.
\]
Since $\bS_{[1^i]}(V) = \bigwedge^i(V)$ if $i=0,1$ and is the cokernel of an injective map $\bigwedge^{i-2}(V) \to \bigwedge^i(V)$ otherwise, we recover the result.

\subsubsection{Free $2$-step nilpotent Lie algebras}

 When $V$ is orthogonal and $\dim(V) = 1$, the Lie algebra $\fI_V(E)$ is $E \oplus \bigwedge^2(E)$, which is the $2$-step truncation of the free Lie algebra on $E$. The homology of this algebra was calculated in several places, see for example \cite{GKT, JW, sigg}. In this case, $\bO(V) = \bZ/2$ and its representations are indexed by the trivial partition $(0)$ and the partition $(1)$. The modification rule calls for removing border strips of length $2\ell(\lambda) - 1$ (at the end, we may have to replace $(0)$ by $(1)$, see \cite[\S 4.4]{lwood} for details, but we may ignore this small point since it does not affect what follows). A simple induction argument shows that if $\lambda = \lambda^\dagger$ is self-dual, then it will reduce to $(0)$ by successively removing such border strips. On the other hand, by working backwards we see that these exhaust all partitions with this property. For the definition of $i_1(\lambda)$, if the border strips we remove are $R_1 ,\dots, R_N$, then $i_1(\lambda) = \sum_i (c(R_i) - 1)$ where $c(R_i)$ is the number of columns of $R_i$. Putting this together, one proves by induction that $i_1(\lambda) = (|\lambda| - \rank(\lambda))/2$ where $\rank(\lambda)$ is the size of the main diagonal of the Young diagram of $\lambda$. So one concludes 
\[
\rH_i(E \oplus \bigwedge^2(E); \bC) = \bigoplus_{\substack{\lambda = \lambda^\dagger\\ |\lambda|+\rank(\lambda)=2i}} \bS_\lambda(E).
\]


\begin{thebibliography}{GKT}

\bibitem[Avr]{avramov} Luchezar L. Avramov, Infinite free resolutions, {\it Six lectures on commutative algebra}, 1--118, Mod. Birkh\"auser Class., Birkh\"auser Verlag, Basel, 2010.

\bibitem[BB]{bjornerbrenti} Anders Bj\"orner, Francesco Brenti, {\it Combinatorics of Coxeter Groups}, Graduate Texts in Mathematics {\bf 231}, Springer, New York, 2005. 

\bibitem[FH]{fultonharris} William Fulton, Joe Harris, {\it Representation Theory: A First Course}, Graduate Texts in Mathematics {\bf 129}, Springer-Verlag, New York, 1991.

\bibitem[Get]{getzler} E. Getzler, The homology groups of some two-step nilpotent Lie algebras associated to symplectic vector spaces, \arxiv{math/9903147v1}.

\bibitem[GKT]{GKT} Johannes Grassberger, Alastair King, Paulo Tirao, On the homology of free 2-step nilpotent Lie algebras, {\it J. Algebra} {\bf 254} (2002), no.~2, 213--225.

\bibitem[JW]{JW} Tadeusz J\'ozefiak, Jerzy Weyman, Representation-theoretic interpretation of a formula of D. E. Littlewood, {\it Math. Proc. Cambridge Philos. Soc.} {\bf 103} (1988), no.~2, 193--196.

\bibitem[KT]{koiketerada} Kazuhiko Koike, Itaru Terada, Young-diagrammatic methods for the representation theory of the classical groups of type $B_n$, $C_n$, $D_n$, {\it J. Algebra} {\bf 107} (1987), no.~2, 466--511.

\bibitem[Kos]{kostant} Bertram Kostant, Lie algebra cohomology and the generalized Borel-Weil theorem, {\it Ann. of Math. (2)} {\bf 74} (1961), 329--387.

\bibitem[Kum]{kumar} Shrawan Kumar, {\it Kac-Moody Groups, their Flag Varieties and Representation Theory}, Progress in Mathematics, 204. Birkh\"auser Boston, Inc., Boston, MA, 2002.

\bibitem[SS1]{expos} Steven~V Sam, Andrew Snowden, Introduction to twisted commutative algebras, \arxiv{1209.5122v1}.

\bibitem[SS2]{infrank} Steven~V Sam, Andrew Snowden, Stability patterns in representation theory, {\it Forum Math. Sigma} {\bf 3} (2015), e11, 108 pp., \arxiv{1302.5859v2}.

\bibitem[SSW]{lwood} Steven~V Sam, Andrew Snowden, Jerzy Weyman, Homology of Littlewood complexes, {\it Selecta Math. (N.S.)}, {\bf 19} (2013), no.~3, 655--698, \arxiv{1209.3509v2}.

\bibitem[SW]{exceptional} Steven~V Sam, Jerzy Weyman, Littlewood complexes and analogues of determinantal varieties, {\it Int. Math. Res. Not. IMRN} (2015), no.~13, 4663--4707, \arxiv{1303.0546v3}.

\bibitem[San]{heisenberg} L.~J. Santharoubane, Cohomology of Heisenberg Lie algebras, {\it Proc. Amer. Math. Soc.} {\bf 87} (1983), no.~1, 23--28.

\bibitem[Sig]{sigg} Stefan Sigg, Laplacian and homology of free two-step nilpotent Lie algebras, {\it J. Algebra} {\bf 185} (1996), no.~1, 144--161.

\bibitem[Wei]{weibel}  Charles A. Weibel, {\it An Introduction to Homological Algebra},  Cambridge Studies in Advanced Mathematics {\bf 38}, Cambridge University Press, Cambridge, 1994.

\bibitem[Wey]{weyman} Jerzy Weyman, {\it Cohomology of Vector Bundles and Syzygies}, Cambridge University Press, Cambridge, 2003.

\end{thebibliography}
\end{document}